\documentclass{amsart}
\usepackage{amssymb}
\usepackage[lite]{amsrefs}

\newcommand{\C}{\mathbb{C}}
\newcommand{\K}{\mathrm{K}}
\newcommand{\Z}{\mathbb{Z}}
\newcommand{\Q}{\mathbb{Q}}
\newcommand{\KK}{\mathrm{KK}}
\newcommand{\red}{\mathrm{r}}
\newcommand{\N}{\mathbb N}
\newcommand{\Comp}{\mathbb K}
\newcommand{\defeq}{:=}

\newtheorem{lemma}{Lemma}

\usepackage[pdftitle={The Baum--Connes conjecture for extensions},
pdfauthor={Ralf Meyer},
pdfsubject={Mathematics}
]{hyperref}

\begin{document}
\author{Ralf Meyer}
\title{The Baum--Connes conjecture for extensions}

\begin{abstract}
  This note provides a counterexample showing that the assumptions
  that Chabert and Echterhoff have imposed in their permanence
  property of the Baum--Connes conjecture for group extensions cannot
  be simplified.
\end{abstract}
\thanks{I thank the referee for very useful suggestions to improve the
  manuscript.}
\maketitle

Chabert and Echterhoff showed in
\cite{Chabert-Echterhoff:Permanence}*{Theorem~3.3} that a discrete
group~$G$ with a normal subgroup~$N$ satisfies the Baum--Connes
conjecture with coefficients in a $G$-C*-algebra~$B$ if and only
if~$G/N$ satisfies the Baum--Connes conjecture with coefficients
$B\rtimes_\red N$ and any subgroup $K\subseteq G$ with $N\subseteq K$
and $[K:N]<\infty$ satisfies the Baum--Connes conjecture with
coefficients~$B$.  Recently, Zhang~\cite{Zhang:BC_extensions} claimed
that it is enough here to assume the Baum--Connes conjecture for~$N$
only instead of for all the subgroups~$K$ above.  This note shows by a
counterexample that this cannot be the case.  So the more complicated
sufficient condition by Chabert and Echterhoff is really needed and
cannot be made simpler.  The group in our case is just a direct
product of a countable discrete group by a finite cyclic group.  So an
intermediate claim about direct products in~\cite{Zhang:BC_extensions}
is also wrong.

It is known that there are counterexamples to the Baum--Connes
conjecture with coefficients.  That is, there is a countable discrete group~$G$
and a separable $G$-C*-algebra such that the Baum--Connes assembly map
for~$G$ with coefficients in~$A$ is not an isomorphism
(see~\cite{Higson-Lafforgue-Skandalis:Counterexamples}).  Let
$D\in \KK^G(P,\C)$ be a Dirac morphism for~$G$ as
in~\cite{Meyer-Nest:BC} and let~$N$ be a mapping cone for it.  As a
consequence, there is a long exact sequence
\begin{multline*}
  \dotsb \to
  \K_{*+1}\bigl((A\otimes N)\rtimes_\red G\bigr)
  \\ \to
  \K_*\bigl((A\otimes P)\rtimes_\red G\bigr) \xrightarrow{D_*}
  \K_*\bigl(A\rtimes_\red G\bigr) \to
  \K_*\bigl((A\otimes N)\rtimes_\red G\bigr) \to \dotsb
\end{multline*}
where the map~$D_*$ induced by~$D$ is the Baum--Connes assembly map.
Then the $\Z/2$-graded group $\K_*\bigl((A\otimes N)\rtimes_\red G\bigr)$
cannot be trivial because this would imply that the Baum--Connes
assembly map is an isomorphism.

\begin{lemma}
  \label{lem:get_p}
  There is a prime~$p$ so that
  $\K_*\bigl((A\otimes N)\rtimes_\red G\bigr)\otimes \Z[1/p]\neq0$.
\end{lemma}

\begin{proof}
  By assumption, there is a nonzero element
  $x\in\K_*\bigl((A\otimes N)\rtimes_\red G\bigr)$.  If its order is
  infinite, we may take~$p$ arbitrary.  If the order of~$x$ is finite,
  then we take~$p$ to be a prime that does not divide the order
  of~$x$.  Then $p^n x\neq 0$ for all $n\in\N$.  The tensor product
  commutes with inductive limits, and
  $\K_*\bigl((A\otimes N)\rtimes_\red G\bigr)\otimes \Z[1/p]$ is the
  inductive limit of the inductive system
  \[
    \K_*\bigl((A\otimes N)\rtimes_\red G\bigr)\xrightarrow{p\cdot {-}}
    \K_*\bigl((A\otimes N)\rtimes_\red G\bigr)\xrightarrow{p\cdot {-}}
    \K_*\bigl((A\otimes N)\rtimes_\red G\bigr)\xrightarrow{p\cdot {-}}
    \dotsb.
  \]
  Therefore, the condition $p^n x\neq 0$ for all $n\in\N$ implies that
  the image of~$x$ in
  $\K_*\bigl((A\otimes N)\rtimes_\red G\bigr)\otimes \Z[1/p]$ is not a zero
  element.  So this group is nontrivial.
\end{proof}

We fix a prime~$p$ as in the lemma.  Let~$\Gamma$ be the cyclic group
of order~$p$.  Actions of the group~$\Gamma$ on the stabilised Cuntz
algebra $B\defeq \mathcal{O}_2 \otimes \Comp(\ell^2\N)$ that belong to
the equivariant bootstrap class in $\KK^\Gamma$ are classified up to
$\KK^\Gamma$-equivalence in~\cite{Meyer:Actions_Kirchberg}.  Since the
$\K$-theory of~$\mathcal{O}_2$ vanishes, this is the rather simple
special case already treated in
\cite{Meyer:Actions_Kirchberg}*{Theorem~7.23}.  The main point for our
purposes here is that $\K_*(B\rtimes_\red \Gamma)$ may be any uniquely
$p$-divisible, $\Z/2$-graded module over the ring \(\Z[\vartheta]\),
where~\(\vartheta\) is a primitive \(p\)th root of unity.  This has
also been proven by Izumi (see
\cite{Izumi:Finite_group}*{Theorem~6.4}).  In particular, there is a
$\Gamma$-action on~$B$ in the equivariant bootstrap class with
\begin{equation}
  \label{eq:B_K-theory}
  \K_0(B\rtimes_\red \Gamma)\cong\Z[1/p,\vartheta],\qquad
  \K_1(B\rtimes_\red \Gamma)\cong 0.
\end{equation}
It follows that $B\rtimes_\red\Gamma$ belongs to the bootstrap class and,
therefore, also satisfies the K\"unneth formula.

Since $\Gamma$ is a finite group, the Dirac morphism
$D\in \KK^G(P,\C)$ is also a Dirac morphism for $G\times \Gamma$ if we
let~$\Gamma$ act trivially on everything.  Therefore, the Baum--Connes
assembly map for $G\times \Gamma$ with coefficients $A\otimes B$ is
the map on $\K$-theory induced by the map
\[
  (A\otimes P\otimes B)\rtimes_\red (\Gamma \times G)
  \xrightarrow{D_*}
  (A\otimes B)\rtimes_\red (\Gamma \times G).
\]
Its mapping cone has the $\K$-theory
\begin{align*}
  \K_*\bigl((A\otimes N\otimes B)\rtimes_\red (\Gamma \times G)\bigr)
  &\cong \K_*\bigl((A\otimes N)\rtimes_\red G \otimes (B\rtimes_\red \Gamma)\bigr)
  \\&\cong \K_*\bigl((A\otimes N)\rtimes_\red G) \otimes \Z[1/p,\vartheta].
\end{align*}
The second isomorphism uses the K\"unneth formula for
$B\rtimes_\red \Gamma$ and that $\Z[1/p,\vartheta]\subseteq \Q$ is
torsionfree.  Since \(\Z[1/p]\) is isomorphic to a direct summand in
\(\Z[1/p,\vartheta]\), Lemma~\ref{lem:get_p} implies that
\(\K_*\bigl((A\otimes N\otimes B)\rtimes_\red (\Gamma \times G)\bigr)
\neq0\).  Thus, the Baum--Connes assembly map for $G\times \Gamma$
with coefficients $A\otimes B$ is not an isomorphism.

However, the Baum--Connes assembly map for $G$ with coefficients
$A\otimes B$ is an isomorphism because the tensor factor~$B$ kills all
$\K$-theory by the K\"unneth theorem.  And the Baum--Connes assembly
map for~$\Gamma$ with any coefficients is an isomorphism
because~$\Gamma$ is finite.  Therefore, we have got a counterexample
for the statement in \cite{Zhang:BC_extensions}*{Theorem~3.13} about
the Baum--Connes conjecture for products of groups.  It is also a
counterexample for the statement in
\cite{Zhang:BC_extensions}*{Theorem~1.2} when we let the normal
subgroup in $\Gamma\times G$ be $\{1\}\times G$.  So both theorems are
false.

\end{document}